\documentclass[10pt,twoside]{article}
\usepackage{amsmath,amsfonts,amssymb,amsthm,graphicx,amscd}
\usepackage{color,cite}

\newtheorem{theorem}{Theorem}
\newtheorem{definition}{Definition}
\newtheorem{remark}{Remark}

\begin{document}

\title{Minimal 3-triangulations of $p$-toroids}

\author{Milica Stojanovi\' c \\Faculty of Organizational Sciences, 154 Jove Ili\' ca Street\\ University of Belgrade\\11040 Belgrade, Serbia \\milica.stojanovic@fon.bg.ac.rs\\}
\date{}

\maketitle

\pagestyle{myheadings} \markboth{\rm M. Stojanovi\'{c}}{\rm Minimal 3-triangulations of $p$-toroids}

\footnote[0]{2000 {\it Mathematics Subject Classifications}.
52C17, 52B05, 05C62.} \footnote[0]{{\it Keywords and Phrases}: triangulation of polyhedra, toroids, piecewise convex polyhedra.}

\begin{abstract}
It is known that we can always 3-triangulate (i.e. divide into tetrahedra) convex polyhedra but not always non-convex ones. Polyhedra topologically equivalent to sphere with $p$ handles, shortly $p$-toroids, could not be convex. So, it is interesting to investigate possibilities and properties of their 3-triangulations. Here, we study the minimal necessary number of tetrahedra for the triangulation of a 3-triangulable $p$-toroid. For that purpose, we developed the concepts of piecewise convex polyhedra and graphs of connection.
\end{abstract}




\section{Introduction}
\label{sec:1}

Dividing polygon by diagonals into triangles is called triangulation. It is known that we can triangulate each polygon with $n$ vertices by $n-3$ diagonals into $n-2$ triangles.

Generalization of this process to higher dimensions is also called triangulation. It consists of dividing polyhedra (polytop) into tetrahedra (simplices) using only original vertices. There are two kinds of problems with triangulation in higher dimensions. It is proved that there is no possibility to triangulate some of non-convex polyhedra \cite{RS}, \cite{Sch} in three-dimensional space, and it is also proved that different triangulations of the same polyhedron may have different numbers of tetrahedra \cite{EPW}, \cite{STT}, \cite{S05}, \cite{S08}. Considering the smallest and the largest number of tetrahedra in triangulation (the minimal and the maximal triangulation), the authors obtained values, which linearly, resp. squarely, depend on the number of vertices. Interesting triangulations are described in the papers of Edelsbrunner, Preparata, West \cite{EPW} and Sleator, Tarjan, Thurston \cite{STT}.

By the term "polyhedron" we usually mean a simple polyhedron, topologically equivalent to sphere. Though there are classes of polyhedra topologically equivalent to torus or $p$-torus (sphere with $p$ handles).

By the definition of Szilassi \cite{Sz-1}, torus-like polyhedra are called toroids. Generalizing that definition, we shall use the term $p$-toroids ($p \in N$ is a given natural number) for $p$-torus-like polyhedra, and term toroids as a common name for all $p$-toroids (the Szilassi's toroids would be called 1-toroids). Since toroids are not convex, it is questionable if it is possible to 3-triangulate them. The 1-toroid with the smallest number of vertices is Cs\' {a}sz\' {a}r polyhedron \cite{BE}, \cite{B-1}, \cite{Cs}, \cite{Sz}, \cite{Sz-1}.
It has 7 vertices and is known to be triangulable with 7 tetrahedra. It is obtained as an example of polyhedron without diagonals \cite{Cs}, \cite{SzS1}, \cite{SzS2}. Some other examples of 1-toroids are given in \cite{Sz}, \cite{Sz-1}, while in \cite{S15}, \cite{S17} 3-triangulations of 1-toroids and 2-toroids are discussed. In \cite{Br}, \cite{JR} some combinatorial properties of $p$-toroids are given.

In Section 2, are described some characteristic polyhedra, while in Section 3 we give some necessary definitions and properties of 3-triangulation of 1-toroids and 2-toroids. In Section 4, we prove that if it is possible to 3-triangulate $p$-toroid with $n$ vertices, then the minimal number of tetrahedra necessary for its 3-triangulation is $T_{min} \geq n + 3(p-1)$. Also, we prove that for each $n \geq 3 + 4p$ there is 3-triangulable $p$-toroid with $T_{min} = n + 3(p-1)$ and discuss whether the same property holds for $p$-toroids with smaller number $n$ of vertices.



\section{Some characteristic examples of polyhedra and their 3-triangulation} \label{sec:2}


{\bf 2.1} Though it is possible to triangulate all convex polyhedra, but this is not the case with non-convex ones. Lennes \cite{L} was the first who present a polyhedron whose interior cannot be triangulated without new vertices. The more famous example, however, was given by Sch\" {o}nhardt \cite{Sch} and referred to in \cite{RS}. Sch\" {o}nhardt's polyhedron is obtained in the following way: triangulate the lateral faces of a trigonal prism $A_1B_1C_1A_2B_2C_2$ by the diagonals $A_1B_2$, $B_1C_2$ and $C_1A_2$. Then "twist" the top face $A_2B_2C_2$ by a small amount in the positive direction. In such a polyhedron, none of tetrahedra with vertices in the set $\{ A_1, B_1, C_1, A_2, B_2, C_2 \}$ is inner, so the triangulation is not possible.


\vspace{0.3cm}
\noindent {\bf 2.2} It is proved that the smallest possible number of tetrahedra in the triangulation of a polyhedron with $n$ vertices is $n - 3$.  An example of polyhedron triangulable with $n - 3$ tetrahedra is a pyramid with $n - 1$ vertices in the basis (i.e., a total of $n$ vertices). We can triangulate it in the following way: we have to do any 2-triangulation of the basis into $(n - 1) - 2 = n - 3$ triangles. Each of these triangles makes with the apex one of tetrahedra in 3-triangulation. 

But, it is not possible to triangulate each polyhedron into $n - 3$ tetrahedra. We shall see later that all triangulations of an octahedron (6 vertices) give 4 tetrahedra.


\vspace{0.3cm}
\noindent {\bf 2.3} Let us now consider triangulations of a bipyramid with $n - 2$ vertices in the basis. The first method is to divide bipyramid into two pyramids and triangulate each of them, taking care of a common 2-triangulation of the basis, then we shall obtain $2(n-4)$ tetrahedra. In the second method, each of $n - 2$ tetrahedra has a common edge joining the apices of the bipyramid, and moreover, each of them contains a pair of the neighbouring vertices of the basis (i.e., one of the edges of the basis). 

If $n = 5$, such a bipyramid has a triangular basis. Then, the first method is  "better", i.e. it gives smaller number of tetrahedra. For $n = 6$ (the octahedron), both methods give 4 tetrahedra, and for $n \geq 7$, the second method is "better". In Figure \ref{fig:1}, triangulations of a bipyramid with a pentagonal basis (i.e. $n=7$) are given. Dividing a bipyramid into two pyramids leads to triangulation with 6 tetrahedra, and dividing it around the axis $V_1V_2$ gives triangulation with 5 tetrahedra.

\begin{figure}[htbp]
\centering
	\includegraphics[width=0.9\textwidth]{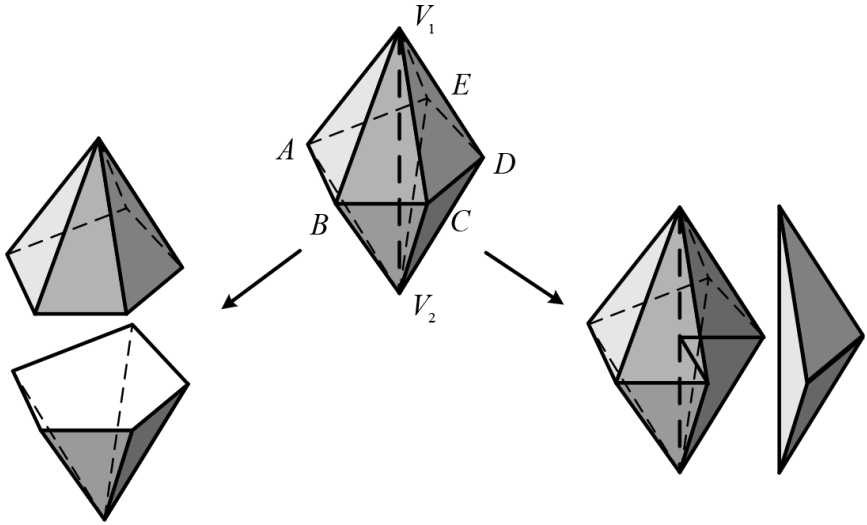}
	\caption{Triangulations of pentagonal bipyramid}
	\label{fig:1}
\end{figure}


\vspace{0.3cm}
\noindent {\bf 2.4} In \cite{Sz-1}, Szilassi introduced a term toroid. Here we shall use the term 1-toroid instead of a toroid.

\begin{definition} \label{d:Sz}(Szilassi) An ordinary polyhedron is called {\em 1-toroid} if it is topologically torus-like (i.e. it can be converted to a torus by continuous deformation) and its faces are simple polygons.
\end{definition}

A 1-toroid with the smallest number of vertices is the Cs\' {a}sz\' {a}r polyhedron (Figure \ref{fig:2}). It has 7 vertices and no diagonals, i.e. each vertex is connected to six others by edges. In \cite{BE} Bokowski and Eggert proved that Cs\' {a}sz\' {a}r polyhedron has four essentially different versions. It is to be noted that in topological terms various versions of  Cs\' {a}sz\' {a}r polyhedron are isomorphic -- there is only one way to draw the full graph with seven vertices on the torus. Szilassi in Wolfram Demonstrations Project \cite{Sz-2} shows that Cs\' {a}sz\' {a}r polyhedron is 3-triangulable with 7 tetrahedra.

\begin{figure}[htbp]
\centering
	\includegraphics[width=0.35\textwidth]{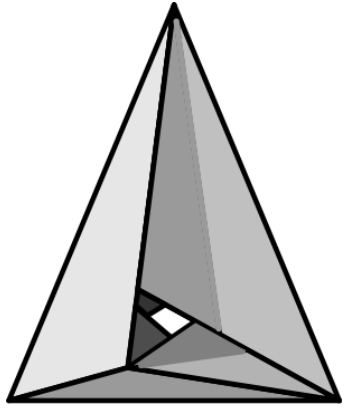}
	\caption{Cs\' {a}sz\' {a}r polyhedron}
	\label{fig:2}
\end{figure}



\section{Preliminaries} \label{sec:3}

In accordance with the definition \ref{d:Sz}, we introduce term $p$-toroid ($p \in N$)

\begin{definition} An ordinary polyhedron is called {\em $p$-toroid}, $p \in N$ is a given number, if it is topologically equivalent to sphere with $p$ handles ($p$-torus) and its faces are simple polygons.
\end{definition}

Let us use term {\em toroid} as a common name for  all $p$-toroids. In our consideration of 3-triangulability of toroids, we shall also use the following definition.

\begin{definition} Polyhedron is {\em piecewise convex} if it is possible to divide it into convex polyhedra $P_i$, $i=1, \ldots , m$, with disjunct interiors. A pair of polyhedra $P_i$, $P_j$ is said to be {\em neighbouring} if they have a common face called {\em contact face}. 
\end{definition}

If the polyhedra $P_i$ and  $P_j$ are not neighbouring, they may have a common edge $e$ or a common vertex $v$. That is possible iff there is a sequence of neighbouring polyhedra $P_i, P_{i+1}, \ldots , P_{i+k} \equiv P_j$ such that the edge $e$, or the vertex $v$ belongs to each contact face ${f_l}$ common to $P_l$ and $P_{l+1}$, $l \in \{ i, \ldots , i+k-1 \}$. Otherwise, polyhedra $P_i$ and  $P_j$ do not have common points.

\begin{remark} \label{r:1} Since it is always possible to 3-triangulate convex polyhedra, the same property holds for piecewise convex polyhedra, especially for piecewise convex toroids. 
\end{remark}

\begin{remark} \label{r:2} Each 3-triangulable polyhedron can be considered as a collection of connected tetrahedra, so such polyhedron is piecewise convex. 
\end{remark}

In Figure \ref{fig:3}, we give an example of 1-toroid $P_9^1$ with $n = 9$ vertices. It is composed of three pieces of convex polyhedra $A$ which are topologically equivalent to triangular prisms. Polyhedron $P_9^1$ is an example of cyclically piecewise convex 1-toroid defined bellow.

\begin{figure}[htbp]
\centering
	\includegraphics[width=0.45\textwidth]{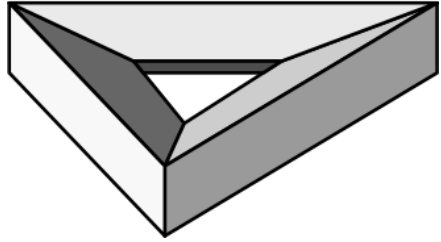}
	\caption{Cyclically piecewise convex polyhedron (1-toroid) $P_{9}^1$}
	\label{fig:3} 
\end{figure}

\begin{definition} 1-toroid is {\em cyclically piecewise convex} if it is possible to divide it into a cycle of convex polyhedra $P_i$, $i=1, \ldots , n$, such that $P_i$ and $P_{i+1}$, $i=1, \ldots , n-1$ and $P_n$ and  $P_1$ are neighbours.
\end{definition}

If a polyhedron $P$ is piecewise convex, let us form a {\em graph of connection} of it, in such a way that {\em nodes} represent convex polyhedra $P_i$, $i=1, \ldots , m$, the pieces of $P$, while {\em edges} represent contact faces between them. 

It is obvious that if 1-toroid is cyclically piecewise convex, then its graph of connection is a single cycle. Other piecewise convex 1-toroids have graphs with a cycle and additional branches. 

Similarly, a piecewise convex $p$-toroid form a planar graph of connection with $p$ cycles, and eventually additional branches. E.g. 2-toroid $P_{14}^2$ given in Figure \ref{fig:4} have two cycles in both of its graphs in Figure \ref{fig:5}. This 2-toroid has $n = 14$ vertices and it is composed of six pieces of $A$ or of two 1-toroids $P_9^1$. It has two graphs of connection because the union of two $A$ parts in the middle form convex polyhedron marked with $2 \cdot A$ in the second of graphs.

This example shows us that division of polyhedron to convex pieces is {\bf not} necessarily unique.

\begin{figure}[htbp]
\centering
	\includegraphics[width=0.58\textwidth]{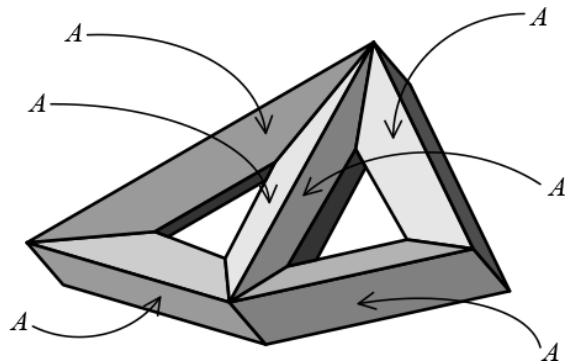}
	\caption{Piecewise convex 2-toroid $P_{14}^2$}
	\label{fig:4} 
\end{figure}

\begin{figure}[htbp]
\centering
	\includegraphics[width=0.68\textwidth]{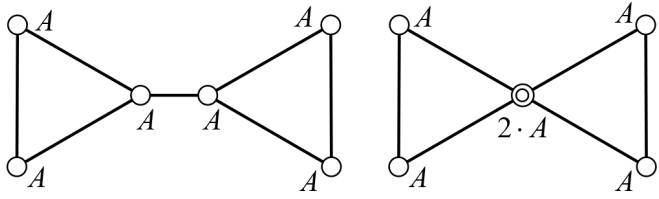}
	\caption{Two graphs of connection for the 2-toroid $P_{14}^2$}
	\label{fig:5} 
\end{figure}

Two different cycles in a graph of 2-toroid can have a common node, or to be connected by an edge or by a branch. The first two cases we observe on the graphs of connection of $P_{14}^2$ in Figure \ref{fig:5}. As in this example, if two cycles of graph for toroid $P$ have a common node, then corresponding cyclically piecewise convex pieces of $P$  share common convex piece, and if they are connected by an edge, they have a contact face. In the third case, two cyclically piecewise convex pieces of $P$ are connected by contact faces with a simple piecewise convex polyhedron inducing branch in the graph of connection. An example of such graph is given in Figure \ref{fig:7} describing 2-toroid $P_{20}^2$ (Figure \ref{fig:6}). $P_{20}^2$ has $n = 20$ vertices and it is composed of two 1-toroids with $n = 10$ vertices connected by polyhedron $A$. In both of the figures convex polyhedron with $n = 7$ vertices is marked with $B$.

\begin{figure}[htbp]
\centering
	\includegraphics[width=0.67\textwidth]{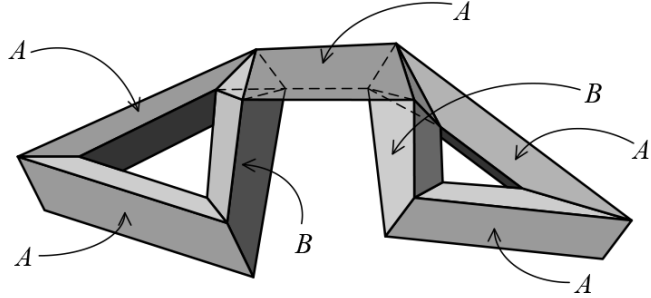}
	\caption{Piecewise convex 2-toroid $P_{20}^2$}
	\label{fig:6} 
\end{figure}

\begin{figure}[htbp]
\centering
	\includegraphics[width=0.55\textwidth]{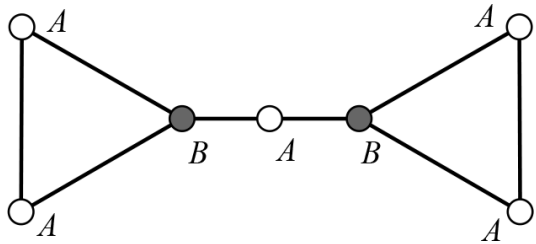}
	\caption{Graph of connection for the 2-toroid $P_{20}^2$}
	\label{fig:7} 
\end{figure}

\pagebreak
In this paper we consider the minimal necessary number of simplices in 3-triangulation of a considered toroid $P$. So, it would be useful to handle with divisions and graphs of toroids in which minimal 3-triangulation of $P$ is in accordance with minimal 3-triangulation of their pieces. Namely, if we do not care about this accordance, it may happen that the sum of tetrahedra in minimal 3-triangulations of pieces would be greater than the number of tetrahedra in 3-triangulation of whole toroid. Really, if contact face of two pieces is with $t \geq 5$ vertices, it may happen that we have around it the bipyramid $R$ with $t$ vertices in the basis. Then, each of two considered pieces would contain one of pyramids as a piece of the bipyramid $R$. As we have observed in 2.3 minimal 3-triangulation of $R$ gives smaller number of tetrahedra then the sum of separate 3-triangulations of pyramids belonging to the pieces. So, let us define:

\begin{definition} \label{d:m} {\em $m$-division} of a polyhedron is a division in which tetrahedra participating in minimal 3-triangulations of pieces are at the same time participating in minimal 3-triangulation of the whole polyhedron. A graph of connection of a given polyhedron is {\em $m$-graph} if it represents $m$-division of that polyhedron.
\end{definition}

\begin{remark} $m$-division and thus $m$-graph of a polyhedron are not unique. Note that convex pieces of division ($m$-division) can be either separate tetrahedra or their different collections. Beside that, there may be more possibilities for minimal 3-triangulation of the same polyhedron.

On the other hand, it is obvious that exists at least one $m$-division of a given 3-triangulable polyhedron. It is its division into tetrahedra participating in minimal 3-triangulation. 
\end{remark}

In \cite{S15} next theorems of 1-toroids are proved:

\begin{theorem} \label{th:3} If it is possible to 3-triangulate 1-toroid with $n \geq 7$ vertices, then the minimal number of tetrahedra necessary for that triangulation is $T_{min} \geq n$.
\end{theorem}

\begin{theorem} \label{th:4} For each $n \geq 7$, there is a 1-toroid which is possible to 3-tri\-an\-gu\-la\-te.
\end{theorem}

Corresponding theorems for 2-toroids are given in \cite{S17}.

\begin{theorem} \label{th:5} If it is possible to 3-triangulate 2-toroid with $n \geq 10$ vertices, then the minimal number of tetrahedra necessary for triangulation is $T_{min} \geq n+3$. 
\end{theorem}
 
\begin{theorem} \label{th:6} For each $n \geq 10$, there is a 2-toroid which is possible to 3-triangulate. 
\end{theorem}



\section{3-triangulations of $p$-toroids} \label{sec:4}

In this section we shall discuss what is the minimal number of tetrahedra necessary for 3-triangulation of 3-triangulable $p$-toroid with $n$ vertices. First, we shall prove the next statement.


\begin{theorem} \label{th:1} If it is possible to 3-triangulate $p$-toroid with $n$ vertices, then the minimal number of tetrahedra necessary for its 3-triangulation is $T_{min} \geq n + 3(p-1)$.
\end{theorem}


There are more possibilities to connect pieces in $m$-graph of $p$-toroid, but in the proof of this theorem it is not necessary to consider them. Here, we shall prove Theorem \ref{th:1} in two different ways using the mathematical induction in both of them.

Note that Theorems \ref{th:3} and \ref{th:5} guarantee that statement is true for $p = 1,2$. Let us suppose (of the following proofs) that statement is true for $p = k$, ($k \in N$) i.e. 


\vspace{2mm}
\begin{minipage}[c]{12cm} {\it If it is possible to 3-triangulate $k$-toroid ($k \in N$) with $n$ vertices, then the minimal number of tetrahedra necessary to its 3-triangulation is $T_{min} \geq n + 3(k-1)$.}
\end{minipage}

\begin{proof} {\it 1.} Observe in $m$-graph $G$ of $(k+1)$-toroid $P$ an edge $e$ which belongs to only one cycle. Such an edge exists because all $m$-graphs are planar. Let us form a new graph $\bar{G}$ with $k$ cycles by excluding $e$ from $G$.
From Definition \ref{d:m} of $m$-graph, it holds that graph $\bar{G}$ is also $m$-graph. Then, appropriate polyhedron $\bar{P}$ is $k$ toroid with $\bar{n}$ vertices. It is obtained from $P$ by "separating" convex pieces and by "duplicating" contact face with $t$ vertices ($t \geq 3$) appropriate to edge $e$ in graph $G$. So, for the number of the vertices of $\bar{P}$ is true 
$$\bar{n} = n + t.$$
Since $\bar{P}$ has $k$ cycles, for its minimal triangulation by induction hypothesis holds
$$T_{min}(\bar{P}) \geq \bar{n} + 3(k-1) = n + t + 3(k-1).$$
Observe that $T_{min}(P) = T_{min}(\bar{P})$. That means 
$$T_{min}(P) \geq n + t + 3(k-1) \geq n + 3\left((k+1) - 1 \right),$$
thus the statement is true for $p = k+1$.
\end{proof}


Note that in this proof the request is to deform a little bit $P$ in order to "separate" convex pieces. In the next proof such deformation is not necessary.

\begin{proof} {\it 2.} In $m$-graph $G$ of $(k+1)$-toroid $P$ observe node $d$ belonging to only one cycle $c$. Let us form a new graph $\hat{G}$ by excluding node $d$, both edges of $G$ containing $d$ which belong to $c$, and also all trees of $G$ containing $d$, if they exist. Such a graph $\hat{G}$ is connected, has $k$ cycles, and it is $m$-graph of $k$-toroid $\hat{P}$ with $n_1$ vertices. Appropriate polyhedron $\hat{P}$ is part of $(k+1)$-toroid $P$, and the rest of $P$ is 0-toroid $\hat{S}$ (i.e. simple polyhedron) with $n_2$ vertices. The pieces $\hat{P}$ and $\hat{S}$ are connected with two contact faces with $t_1$ and $t_2$ ($t_1, t_2 \geq 3$) vertices resp., so
$$n = n_1 + n_2 - (t_1 + t_2).$$
Using induction hypothesis we obtain
\begin{eqnarray*}
T_{min}(P) & = & T_{min}(\hat{P}) + T_{min}(\hat{S}) \geq \\
		& \geq & (n_1 + 3(k-1)) + (n_2-3) = \\
		& = & n_1 + n_2 +3(k-2) = \\
		& = & n + t_1 + t_2 + 3(k-2) \geq \\
		& \geq & n + 3\left((k+1) - 1 \right).
\end{eqnarray*}
\end{proof}

Considering the smallest number $n$ of vertices in 3-triangulable $p$-toroid it will be necessary to take care about different possibilities of connecting pieces in $m$-graph. Note that some polyhedron might be topologically (combinatorially) realizable but not also geometrically. That is reason to create more examples of topologically realizable $p$-toroids. Checking if their geometric realizations exist will be left for some future paper. For the series of $p$-toroids described in the proof of the following theorem it is obvious that it is geometrically realizable.


\begin{theorem} \label{th:2} For each $n \geq 4p + 3$ there is 3-triangulable $p$-toroid with $T_{min} = n + 3(p-1)$.
\end{theorem}


\begin{proof} First we shall form main series of $p$-toroids $\bar{P}^p_{4p+3}$  by gluing $p$ Cs\' {a}sz\' {a}r's  toroids into chain. Each pair of neighbour 1-toroids have a common contact face. $m$-graphs of these $p$-toroids are formed of $p$ heptagons connected by $p-1$ edges, as it is shown in the Figure \ref{fig:8}.

\begin{figure}[htbp]
\centering
	\includegraphics[width=0.8\textwidth]{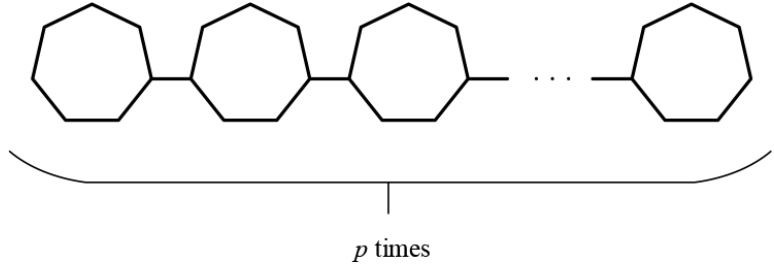}
	\caption{Graph of connection for $p$-toroid $\bar{P}^p_{4p+3}$}
	\label{fig:8} 
\end{figure}

In $p$-toroids $\bar{P}^p_{4p+3}$ neighbour 1-toroids have 3 common vertices, so the total number of vertices of $\bar{P}^p_{4p+3}$ is $n = 4p + 3$. On the other hand, the number of tetrahedra in the 3-triangulation of $\bar{P}^p_{4p+3}$ is equal to $7p$, i.e. $T_{min} (\bar{P}^p_{4p+3}) = 7p$. Since for $n = 4p +3$ holds
$$n + 3(p-1) = (4p +3) + 3(p-1) = 7p = T_{min} (\bar{P}^p_{4p+3}),$$
claim is true whenever $n = 4p +3$.

If $n > 4p + 3$ we can take any simple polyhedron $S_k$ with $k = n - (4p + 3) + 3 = n - 4p$ vertices, which has triangular faces and $T_{min} (S) = k - 3$, e.g. a pyramid with space $(k-1)$-gon in the basis. Then $p$-toroid $\bar{P}^p_n$ can be formed by gluing $p$-toroid $\bar{P}^p_{4p+3}$ and $S_k$ so that they have a common triangular face. Then, the number of vertices of $\bar{P}^p_n$ is
$$(4p+3) + (n - 4p) - 3 = n,$$
and 
\begin{eqnarray*}
   T_{min} (\bar{P}^p_n) & = & T_{min} (\bar{P}^p_{4p+3}) + T_{min} (S_k) \\
   											 & = & 7p + k - 3 \\
   											 & = & 7p + n - 4p - 3 \\
   											 & = & n + 3p - 3 \\
   											 & = & n + 3(p-1).
\end{eqnarray*}
\end{proof}

A smaller number of vertices in $p$-toroid appears if in main series of $p$ Cs\' {a}sz\' {a}r's  toroids, neighbour ones have a common tetrahedron instead of a common face (Figure \ref{fig:9}). Then number $n$ of vertices in such $\hat{P}^p_{3p+4}$ is 
$n = 3p+4$, while $T_{min} (\hat{P}^p_{3p+4}) = 6p + 1$. Since $6p + 1 = (3p+4) + (3p - 3) = n + 3(p-1)$, holds
$$T_{min} (\hat{P}^p_{3p+4}) = n + 3(p-1).$$

\begin{figure}[htbp]
\centering
	\includegraphics[width=0.62\textwidth]{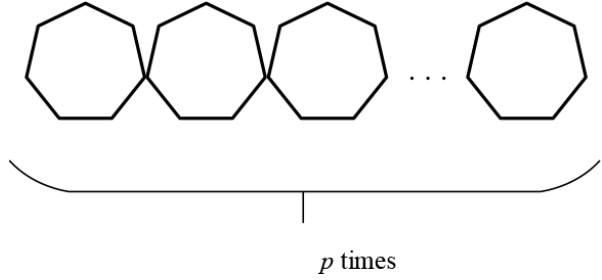}
	\caption{Graph of connection for $p$-toroid $\hat{P}^p_{3p+4}$}
	\label{fig:9} 
\end{figure}

But, for $\hat{P}^p_{3p+4}$ it is questionable if it has geometric realization. In \cite{S17} double-Cs\' {a}sz\' {a}r 2-toroid, which is $\hat{P}^2_{10}$ from this series was introduced. It was proved that it is 3-triangulable 2-toroid with the smallest number of vertices, $n = 10$. It seems likely that this 2-toroid has geometric realization. Next toroid in this series is $\hat{P}^3_{13}$, composed of three Cs\' {a}sz\' {a}r's 1-toroids. Since all three Cs\' {a}sz\' {a}r's 1-toroids which build $\hat{P}^3_{13}$ have at least one common vertex, geometric realization of $\hat{P}^3_{13}$ is not so obvious. 

On the other hand, if this chain is geometrically realizable, we can think if it is possible for some enough great $p$ to close new circle to form $\tilde{P}^{p+1}_{3p}$ (Figure \ref{fig:10}). Then this `cycle of cycles' would be $(p+1)^{\mathrm {th}}$ cycle, the number of vertices would be $n = 3p$, and again
\begin{eqnarray*}
   T_{min} (\tilde{P}^{p+1}_{3p}) & = & T_{min} (\hat{P}^p_{3p+4}) - 1 = \\
    															& = & 6p \\
    															& = & n + 3\left((p+1) - 1 \right).
\end{eqnarray*}

\begin{figure}[htbp]
\centering
	\includegraphics[width=0.44\textwidth]{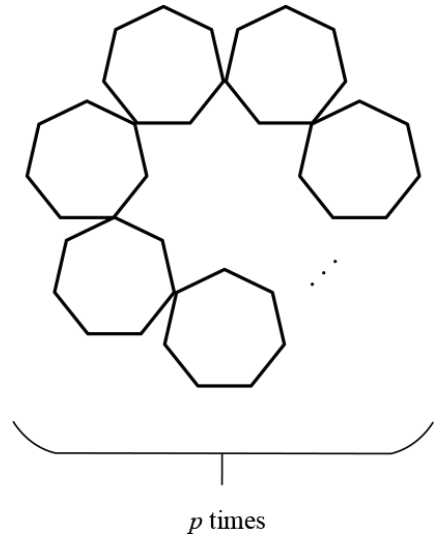}
	\caption{Graph of connection for $(p+1)$-toroid $\tilde{P}^{p+1}_{3p}$}
	\label{fig:10} 
\end{figure}

Of course, it is also possible to think of closing cycle of Cs\' {a}sz\' {a}r's 1-toroids in a chain of type $\bar{P}^p_{4p+3}$. Such $(p+1)$-toroid would have $n = 4p$ vertices and $T_{min}$ would be equal to $7p$. Again it would be $7p = n + 3\left((p+1) - 1 \right)$.

Note that the smallest possible number of vertices for $p$-toroid is considered in \cite{JR}, in a combinatorial way. It is proved for example, that the minimal number of vertices for 2-toroids and 3-toroids is $n = 10$. Geometrical realization is not considered in this paper. Also, it is not known if mentioned toroids are 3-triangulable. In \cite{Br} is proved that 3-toroid with $n = 10$ vertices have geometrical realization, but it remains to investigate its 3-triangulability.





\end{document}